\pdfoutput=1

\documentclass{amsart}

\usepackage{amsfonts,amssymb,verbatim,amsmath,amsthm,latexsym,textcomp,amscd}
\usepackage{latexsym,amsfonts,amssymb,epsfig,verbatim}
\usepackage{amsmath,amsthm,amssymb,latexsym,graphics,textcomp,hyperref}
\usepackage{graphicx}
\usepackage{color}
\usepackage{url}
\usepackage{psfrag}
\usepackage{amsmath,comment,amscd}
\usepackage{amssymb}
\usepackage{amsthm}

\usepackage{caption}
\usepackage{color}
\usepackage{enumerate}
\usepackage{float}
\usepackage[a4paper]{geometry}
\usepackage{graphicx}
\usepackage{tikz-cd}

\numberwithin{equation}{section}

\newcommand{\define}{\textbf}

\newcommand{\mc}{\mathcal}

\newcommand{\vp}{\varphi}

\newcommand{\ab}[1]{\langle#1\rangle} 

\newcommand{\floor}[1]{\lfloor#1\rfloor}
\newcommand{\lrb}[1]{\left[#1\right]}

\newcommand{\lrp}[1]{\left(#1\right)}

\newcommand{\lrset}[1]{\left\{#1\right\}}
\newcommand{\norm}[1]{\|#1\|}
\newcommand{\set}[1]{\{#1\}}
\newcommand{\suppressthis}[1]{}

\newcommand{\twopartdef}[4]
{
	\left\{
	\begin{array}{ll}
		#1 & \mbox{if } #2 \\
		#3 & \mbox{if } #4
	\end{array}
	\right.
}

\newcommand{\R}{\mathbf R}

\newcommand{\T}{\mathbb T}
\newcommand{\Z}{\mathbb Z}
\newcommand{\C}{\mathbf C}

\newtheorem{theorem}{Theorem}[section]
\newtheorem{lemma}[theorem]{Lemma}

\newtheorem{corollary}[theorem]{Corollary}

\DeclareMathOperator{\ann}{Ann}
\DeclareMathOperator{\dir}{Dir}
\DeclareMathOperator{\ord}{ord}
\DeclareMathOperator{\supp}{supp}

\newcommand{\one}{\mathbf 1}

\begin{document}

\title{On Configurations of Order 2}
\author{Abhishek Khetan}
\address{School
	of Mathematics, Tata Institute of Fundamental Research. 1, Homi Bhabha Road, Mumbai-400005, India}

\email{khetan@math.tifr.res.in}



\date{\today}

\begin{abstract}
	Let $c:\Z^2\to \set{0, 1}$ be a configuration with a non-trivial annihilator.
	We show that if $c$ is weakly periodic then the directions of periodicity in a minimal weakly periodic decomposition of $c$ can be detected from the annihilator ideal associated to $c$.
	We show that the order of a weakly periodic configuration is same as the number of components in any minimal decomposition into $1$-periodic elements.
	We then give an upper bound on the order in terms of the support of any of its annihilators.
   In the special case of tilings this gives an upper bound on the order of any tiling in terms of a geometric quantity associated to the tile.
   We prove that if $c:\Z^2\to \set{0, 1}$ is a configuration having a non-trivial annihilator and has order $2$ then it can be written as a sum of two $1$-periodic configurations valued in $\set{0, 1}$.
   Lastly we show that any tiling of $\Z^2$ by a tile of cardinality the square of a prime has a point of order at most $2$ in its orbit closure.
\end{abstract}

\maketitle

\tableofcontents

\section{Introduction}

A \emph{configuration} is a function $c:\Z^2\to \C$.
Given a configuration $c$, one way to measure its complexity is by probing $c$ with a finite shape.
More precisely, if $D\subseteq \Z^2$ is a finite set, then $c$ induces a function $D\to \C$ given any translation of $D$.
The set of all such functions obtained this way is written as $\mc P_c(D)$.
The size of $\mc P_c(D)$ is a measure of the `local complexity' of $c$.
We say that $c$ has \emph{low (local) complexity} with respect to $D$ if $|\mc P_c(D)|\leq |D|$.
On the other hand, we say that $c$ is $1$-\emph{periodic} if there is a nonzero vector in $\Z^2$ such that translating $c$ by this vector does not alter $c$.
Thus, if $c$ is $1$-periodic, then, essentially, $c$ is a $1$-dimensional configuration, rather than a genuine $2$-dimensional configuration.
Thus, a $1$-periodic configuration can be thought of as a configuration which has low `global complexity.'
It is natural to expect a relationship between the local and the global notions of complexities.
\emph{Nivat's conjecture} formalizes this by stating that if $c$ has low complexity with respect to the rectangle $D=\set{1, \ldots, m}\times \set{1, \ldots, n}$, then $c$ is $1$-periodic.
The conjecture as stated is still open, however, a lot of progress has been made by various authors (see \cite{cyr_kra}, \cite{kari_szabados_alg_geom}, \cite{kari_mutot}).
An algebraic measure of complexity on a configuration was introduced in \cite{kari_szabados_alg_geom}, where, roughly, it was shown that if a configuration satisfies a certain algebraic condition then there is associated to this configuration a set of directions, and the number of these directions is called the \emph{order} of the configuration.
In \cite{szabados_nivat_conjecture} it was shown that Nivat's conjecture indeed holds if one additionally assumes that the configuration has order at most $2$ (that is, the configuration is algebraically not very complex).
Using this result Kari and Szabados \cite{kari_szabados_alg_geom} recovered the result of \cite{cyr_kra} using different methods.

It is thus natural to ask what can be said about periodicity of configurations which have `low algebraic complexity' without any insistence on low local complexity.
Theorem \ref{theorem:periodicity of binary configuration} we show that if $c$ is a configuration valued in $\set{0, 1}$ and has order at most $2$, then $c$, when thought of as a subset of $\Z^2$, can be partitioned into two $1$-periodic subsets.
We then show that natural generalizations of this fact to configurations valued in $\set{0, 1, 2}$ or of higher order do not hold.
The following is a summary of the results in this paper.

Section \ref{section:preliminaries} concerns with the definitions of the main concepts around which this paper is centered.
We also discuss some results developed by \cite{kari_szabados_alg_geom} which are needed in Section \ref{section:possible directions of one periodicity} and \ref{section:order two configuration}.
In this paper we focus on \emph{binary configurations} which are nothing but subsets of $\Z^2$.
Thus \emph{tilings} (see \S \ref{section:preliminaries} for a definition) are examples of binary configurations.
Following \cite{kari_szabados_alg_geom} one may associate the ideal of annihilators in $\C[x^\pm, y^\pm]$ to any given binary configuration which can be used to deduce structural information about the given configuration if this ideal is non-trivial.
Let $c$ be a binary configuration with a non-trivial annihilator.
By a result of \cite{kari_szabados_alg_geom} it follows that this annihilator ideal is of the form $\ab{\phi_1 \cdots \phi_m}H$ where each $\phi_i$ is a \emph{line polynomial} (See Section \ref{section:preliminaries} for definition) and $H$ is the intersection of finitely many maximal ideals.
To each line polynomial there is associated a unique direction.
We show that if $c$ is $1$-periodic then it has a periodicity vector parallel to the direction of one of the $\phi_i$'s.
In fact, we prove a more general result in Section \ref{section:possible directions of one periodicity} in the form of Lemma \ref{lemma:order and directions of weakly periodicity} where, roughly speaking, we show that the directions of $1$-periodic components in a minimal \emph{weakly periodic}\footnote{A binary configuration is weakly periodic if it can be partitioned into finitely many $1$-periodic subsets.} decomposition of any binary configuration with a non-trivial annihilator are the directions of the line polynomials $\phi_1, \ldots, \phi_m$.
Lemma \ref{lemma:order and directions of weakly periodicity} also shows that the order (See Theorem \ref{theorem:annihilator theorem kari szabdos} for definition) of a weakly periodic point is same as the number of components in any minimal decomposition into $1$-periodic elements.

Further, Lemma \ref{lemma:directions lemma} shows that the size of the \emph{direction set} (See Section \ref{section:possible directions of one periodicity}) of the support of any annihilator is an upper bound on its order. 
Thus the direction set provides structural information, for instance, if the order of a configuration is $1$ then the configuration is $1$-periodic.
If $T\subseteq \Z^2$ is a tiling by translates of a finite set $F$ (called a \emph{cluster}) then Lemma \ref{lemma:order of a tiling} shows that the order of any $F$-tiling is at most the size of the direction set of $F$.
This in particular shows that if the direction set of a cluster $F$ is a singleton then each $F$-tiling is $1$-periodic.

Let $c$ be a binary configuration with a non-trivial annihilator.
Assume the order of $c$ is $2$.
Then Theorem \ref{theorem:periodicity of binary configuration} shows that $c$ is weakly periodic and has two components in any minimal weakly periodic decomposition.
Natural generalizations of this result do not hold, even up to orbit closure, and we give counterexamples in Section \ref{subsection:counterexample for finitary configurations} and \ref{subsection:counterexample for binary configurations}.

Lastly, Theorem \ref{theorem:prime square theorem} shows that if $F$ is a cluster in $\Z^2$ of size $p^2$, where $p$ is a prime, and $T$ is an $F$-tiling of $\Z^2$, then there is a tiling in the orbit closure of $T$ having order at most $2$.
The proof of this result takes the analytic route based on the ergodic theoretic proof of the periodic tiling conjecture by Siddhartha Bhattacharya in \cite{bhattacharya_tilings}.

In Section \ref{subsection:order versus support of spectral measure} we show a connection between the algebraic methods of Kari and Szabados with the analytic methods of Bhattacharya.
In \cite[Section 3]{bhattacharya_tilings} Bhattacharya shows that the support of the spectral measure arising from a binary configuration (with a non-trivial annihilator) is contained in the union of finitely many one dimensional sub-torii of the two dimensional torus.
The number of these subtorii is at most the order of the configuration and, in fact, the directions of these sub-torii is contained in the directions of the line polynomials appearing in the decomposition theorem of Kari and Szabados.
This connection, although not necessary for the proofs of any of the results mentioned above, has independent interest in the opinion of the author, as it sheds light on the underlying similarity of the various techniques developed to attack this circle of problems.

\subsection{Acknowledgements}
The author would like to thank Mahan Mj, Siddhartha Bhattacharya, Ankit Rai, Pierre Guillon, Etienne Mutot, and Nishant Chandgotia for helpful conversations and their encouragement.
Special thanks go to Nishant Chandgotia for a careful reading of various parts of the manuscript and many helpful comments and his guidance.

%
%
%
%

%
\section{Preliminaries}
\label{section:preliminaries}

\subsection{Tilings}

Let $F\subseteq \Z^2$ be finite, which we will refer to as a \define{cluster}.
We define an $F$-\define{tiling} as a subset $T$ of $\Z^2$ such that
\begin{equation}
	\sum_{v\in F} 1_T(p-v) = 1
\end{equation}
for all $p\in \Z^2$.
Thus an $F$-tiling as a partition of $\Z^2$ by translates of $F$.
We say that a cluster $F$ is \define{exact} if there exists an $F$-tiling.

\subsection{Configurations}

A \define{configuration} is a function $c:\Z^2\to \C$.
A configuration taking values only in $\Z$ is called \define{integral} and a configuration taking only finitely many values is called \define{finitary}.
A configuration valued in $\set{0, 1}$ will be referred to as a \define{binary configuration}.
Of course, any subset of $\Z^2$ can be thought of as a binary configuration, and hence tilings, in particular, serve as examples of binary configurations.

We say that $x\in \set{0, 1}^{\Z^2}$ is $1$-\define{periodic} if there is a nonzero vector $v\in \Z^2$ such that $v\cdot x = x$.\footnote{Here $v \cdot x$ denotes the configuration which takes $u$ to $x(u-v)$.}
We say $x$ is \define{biperiodic} if there is a finite index subgroup $\Lambda$ of $\Z^2$ such that $v\cdot x=x$ for all $v\in \Lambda$.
Note that $x$ is biperiodic if and only if there exist two linearly independent vectors $u$ and $v$ in $\Z^2$ such that $u\cdot x = v\cdot x = x$.
A point $x\in \set{0, 1}^{\Z^2}$ is called \define{weakly periodic} if there exist finitely many $1$-periodic points $x_1, \ldots, x_k\in \set{0, 1}^{\Z^2}$ such that $x=x_1+ \cdots + x_k$.
Thus $x$ is weakly periodic if and only if it, when thought of as a subset of $\Z^2$, can be partitioned into finitely many $1$-periodic subsets.

\subsection{Algebraic Concepts}

Let $\C[x^\pm, y^\pm]$ denote the set of all the Laurent polynomials in two variables.
For $v_1, v_2\in \Z$, we will write the monomial $x^{v_1}y^{v_2}$ as $U^v$, where $v=(v_1, v_2)\in \Z^2$.
Thus we may denote $\C[x^\pm, y^\pm]$ as $\C[U^\pm]$.
A typical element of $\C[U^{\pm}]$ is written as $\sum_{v\in \Z^2} a_v U^v$ where only finitely many of the $a_v$'s are non-zero.
On the other hand, we write $\C[[U^\pm]]$ to denote the set of all formal sums $\sum_{v} a_v U^v$ where $a_v\in \C$ are arbitrary.
An element of $\C[[U^\pm]]$ will be referred to as a \define{Laurent series}.
Any configuration can be naturally thought of as a Laurent series.
Even though one cannot multiply two Laurent series, it is still meaningful to multiply a Laurent polynomial with a Laurent series.
Henceforth, we will simply use the word `polynomial' to mean a Laurent polynomial. 

We say that a Laurent polynomial $f$ \define{annihilates} a configuration $c$ if $f\cdot c = 0$.
Thus a configuration $c$ is $1$-periodic if and only if there is a nonzero vector $v\in \Z^2$ such that $U^v-1$ annihilates $c$ and biperiodic if and only if there are two linearly independent vectors $v$ and $w$ such that $U^v-1$ and $U^w -1$ both annihilate $c$.
The set of all the Laurent polynomials which annihilate a configuration $c$ forms an ideal in the ring $\C[U^\pm]$.
%
The following is a somewhat technical result which we use to establish Lemma \ref{lemma:directions lemma}.
\begin{theorem}
	\label{theorem:more annihilators from a given annihilator}
	\emph{\cite[Theorem 12]{kari_szabados_alg_geom}}
	Let $c$ be a finitary integral configuration and let $f=\sum_{v} a_v U^v$ be a non-trivial integer polynomial in $\ann(c)$.
	Then there is a positive integer $n_0$ depending only on $f$ such that for any $v_0\in \supp(f)$ we have
	\begin{equation}
		\prod_{v\in \supp(f), v\neq v_0} (U^{n_0(v-v_0)} - 1)
	\end{equation}
	is also in $\ann(c)$.
\end{theorem}

A \define{line polynomial} is a Laurent polynomial of the form
\begin{equation}
	\sum_{i=-\infty}^\infty a_i U^{iv}
\end{equation}
where $v$ is a nonzero vector (and only finitely many of the $a_i$'s are nonzero).
The following theorem is a key result about finitary integral configurations and makes apparent the power of the algebraic methods of Kari and Szabados in studying structural properties of configurations.

\begin{theorem}
	\label{theorem:annihilator theorem kari szabdos}
	\emph{\cite[Theorem 24, Corollary 28]{kari_szabados_alg_geom}}
	Let $c$ be a finitary integral configuration with a non-trivial annihilator.
	Then there are line polynomials $\phi_1 , \ldots, \phi_m$ and an ideal $H$ in $\C[U^\pm]$ such that
	\begin{enumerate}[a)]
		\item Each $\phi_i$ is a line polynomial with the property that there is a primitive vector $v_i$ in $\Z^2$ such that $\phi_i$ is the product of polynomials of the form $U^{v_i}-\omega$, where $\omega$ is a root of unity.
		\item $H$ is an intersection of finitely many maximal ideals in $\C[U^\pm]$.
		\item $\ann(c)=\ab{\phi_1 \cdots \phi_m}H$.
		\item The directions of $\phi_i$'s are pairwise distinct.
		\item The $\phi_i$'s and $H$ are unique up to multiplication by invertible elements and rearrangements.
		\item The ideals $\ab{\phi_1 \cdots\phi_m}$ and $H$ are comaximal.
	\end{enumerate}
\end{theorem}
The `$m$' appearing in the above description is called the \define{order} of $c$.

\begin{lemma}
	\label{lemma:strong radicality for line polynomials}
	\emph{\cite[Lemma 10]{kari_szabados_alg_geom}}
	Let $c$ be a finitary configuration and let $f_1, \ldots, f_k$ be line polynomials such that $f_1^{m_1} \cdots f_k^{m_k}$ annihilates $c$.
	Then $f_1 \cdots f_k$ also annihilates $c$.
\end{lemma}

The following is a fundamental structural result about configurations.
\begin{theorem}
	\label{theorem:kari szabados decomposition theorem}
	\cite[Corollary 28]{kari_szabados_alg_geom}
	Let $c$ be a finitary integral configuration with a non-trivial annihilator.
	If $c$ has order $m$, then there are $1$-periodic configurations $c_1, \ldots, c_m:\Z^2\to \C$ and a biperiodic configuration $c':\Z^2\to \C$ such that
	\begin{equation}
		c = c_1 + \cdots + c_m + c'
	\end{equation}
\end{theorem}

\subsection{Spectral Theorem}
\label{section:spectral theorem}

Let $\T^2$ be the two dimensional torus.
Let $\nu$ be a probability measure on $\T^2$.
There is a canonical representation of $\Z^2$ on $L^2(\T^2, \nu)$ which we now describe.
Recall that the characters on $\T^2$ are in bijection with $\Z^2$.
Let us write $\chi_g:\T^2\to \C$ to denote the character corresponding to $g\in \Z^2$.
Then we define a map $\sigma_g:L^2(\T^2, \nu)\to L^2(\T^2, \nu)$ by writing $\sigma_g(\phi) = \chi_g\phi$, where the latter is the pointwise product of $\chi_g$ and $\phi$.
It can be easily checked that each $\sigma_g$ is in fact a unitary linear map.
Thus we get a map $\sigma:\Z^2\to \mc U(L^2(\T^2, \nu))$ which takes $g$ to $\sigma_g$.

By the Stone-Weierstrass theorem we have the $\C$-span of the characters are dense in $C(\T^2)$, where $C(\T^2)$ is the set of all the complex valued continuous functions on $\T^2$ equipped with the sup-norm topology.
Also, since $\T^2$ is a compact metric space, we have $C(\T^2)$ is dense in $L^2(\T^2, \nu)$.
Therefore the $\C$-span of the characters are dense in $L^2(\T^2, \nu)$.
From this we see that, if $\one$ denotes the constant map which takes the value $1$ everywhere, $\text{Span}\set{\sigma_g\one:\ g\in \Z^2}$ is dense in $L^2(\T^2, \nu)$.
In other words, $\one$ is a \emph{cyclic vector} for this representation.
We now want to state a theorem which dictates that this is a defining property of unitary representations of $\Z^2$.

\begin{theorem}
	\label{theorem:spectral theorem}
	\textit{\textbf{Spectral Theorem.}}
	Let $H$ be a Hilbert space and $\tau:\Z^2\to \mc U(H)$ be a unitary representation of $\Z^2$.
	Suppose $v\in H$ is a cyclic vector, that is, $\text{Span}\set{\tau_gv:\ g\in \Z^2}$ is dense in $H$, and assume that $v$ has unit norm.
	Then there is a unique probability measure $\nu$ on $\T^2$ and a unitary isomorphism $\theta: H \to L^2(\T^2, \nu)$ with $\theta(v)= \one$ such that the following diagram commutes for all $g\in \Z^2$
	\begin{figure}[H]
		\centering
		\begin{tikzcd}
			H\ar[r, "\theta"]\ar[d, "\tau_g"] &	L^2(\T^2, \nu)\ar[d, "\sigma_g"]\\
			H\ar[r, "\theta"]\ar[r, "\theta"] & L^2(\T^2, \nu)
		\end{tikzcd}
	\end{figure}
	\noindent
\end{theorem}
So the above theorem says that the abstract representation $\tau$ can be thought of as the canonical concrete representation $\sigma$, at the cost of a mysterious probability measure $\nu$.
Thus understanding the measure $\nu$ is equivalent to understanding $\tau$.

\section{Possible Directions of Weak Periodicity}
\label{section:possible directions of one periodicity}

\begin{lemma}
	\label{lemma:order and directions of weakly periodicity}
	Let $c:\Z^2\to \set{0, 1}$ have a non-trivial annihilator.
	Assume that $c$ is weakly periodic but not biperiodic.
	Let $k$ be the smallest positive integer such that there exist $1$-periodic binary configurations $c_1, \ldots, c_k$ such that $c=c_1 + \cdots + c_k$.
	Then $k=\ord(c)$.
	Further, each $c_i$ is periodic in the direction of some $\phi_j$.
\end{lemma}
\begin{proof}
	By Theorem \ref{theorem:annihilator theorem kari szabdos} we know that $\ann(c) = \ab{\phi_1 \cdots \phi_m}H$ where each $\phi_i$ is a line polynomials in pairwise distinct directions and $H$ is the product of finitely maximal ideals.
	Let $g_1, \ldots, g_k$ be nonzero vectors in $\Z^2$ such that $g_i\cdot c_i=c_i$ for $i=1, \ldots, k$.
The minimality of $k$ implies that the $g_i$'s are pairwise linearly independent.
By Theorem \ref{theorem:annihilator theorem kari szabdos} we know that each $\phi_i$ is a product of polynomials of the form $U^{v_i} - \omega_{ij}$ where $v_i$ is a primitive vector in $\Z^2$, $\omega_{ij}$ is a root of unity.
The $v_i$'s are pairwise linearly independent.
Also, the $\phi_i$'s and $H$ are uniquely determined.

Now since $g_i$ fixes $c_i$, we have $f(U) := \prod_{i=1}^k (U^{g_i} - 1)\in \ann(c)$.
Thus for any $i\in \set{1, \ldots, m}$ we have a polynomial of the form $U^{v_i} - \omega_{ij}$ divides $f(U)$, since the irreducible factors of $\phi_i$ are of this form.
This shows that each $v_i$ is in the direction of some $g_j$ and therefore $k\geq m$.
This also shows that each $\phi_i$ divides some power of some $U^{g_j}-1$.
By relabeling the $g_j$'s if needed, we may assume that $\phi_i$ divides some power of $U^{g_i}-1$ for $i=1, \ldots, k$.

Assume that $k>m$.
We will produce a contradiction.
Say $H= \prod_{i=1}^l \ab{x-\alpha_i, y-\beta_i}$ for some complex numbers $\alpha_1, \ldots, \alpha_l, \beta, \ldots, \beta_l$.
Then
\begin{equation}
	\phi_1 \cdots \phi_m (x-\alpha_1) \cdots (x-\alpha_l) \in \ann(c)
\end{equation}
Thus
\begin{equation}
	[(x-\alpha_1) \cdots (x-\alpha_l) (U^{g_1} - 1) \cdots (U^{g_m} - 1)] (c_1 + \cdots + c_k) = 0
\end{equation}
where we have appealed to Lemma \ref{lemma:strong radicality for line polynomials} to get rid of powers of $U^{g_i}-1$.
This gives
\begin{equation}
	[(x-\alpha_1) \cdots (x-\alpha_l) (U^{g_1} - 1) \cdots (U^{g_m} - 1)] (c_{m+1} + \cdots + c_k) = 0
\end{equation}
So the configuration $c':= c_{m+1} + \cdots +c_k$ is finitary integral and has a non-trivial annihilator.
Using Theorem \ref{theorem:annihilator theorem kari szabdos} we have $\ann(c') = \ab{\psi_1 \cdots \psi_t} H'$ where $\psi_i$'s are line polynomials in pairwise different directions and $H'$ is the intersection of finitely many maximal ideals.
Note that the minimality of $k$ implies that $c'$ is not biperiodic.

In the reasoning above we may replace $(x-\alpha_1) \cdots (x-\alpha_l)$ with $(y-\beta_1) \cdots (y-\beta_l)$ to get 
\begin{equation}
	(x-\alpha_1) \cdots (x-\alpha_l) (U^{g_1} - 1) \cdots (U^{g_m}-1) \text{ and } (y-\beta_1) \cdots (y-\beta_l) (U^{g_1} - 1) \cdots (U^{g_m} - 1)
\end{equation}
are in $\ann(c')$.
Fix $j\in \set{1, \ldots, t}$.
Let $w_j$ be a primitive vector such that $U^{w_j} -\eta$ is an irreducible factor of $\psi_j$, where $\eta$ is a root of unity.
Now we have
\begin{equation}
	U^{w_j} - \eta \text{ divides } (x-\alpha_1)  \cdots (x-\alpha_l)(U^{g_1} - 1) \cdots (U^{g_m} - 1)
\end{equation}
and thus
\begin{equation}
	U^{w_j} - \eta \text{ divides } (x-\alpha_1)  \cdots (x-\alpha_l)\quad  \text{ or }\quad  U^{w_j} - \eta \text{ divides } (U^{g_1} - 1) \cdots (U^{g_m} - 1)
\end{equation}
Similarly
\begin{equation}
	U^{w_j} - \eta \text{ divides } (y-\beta_1)  \cdots (y-\beta_l)\quad  \text{ or }\quad  U^{w_j} - \eta \text{ divides } (U^{g_1} - 1) \cdots (U^{g_m} - 1)
\end{equation}
This implies that
\begin{equation}
	U^{w_j} - \eta \text{ divides } (U^{g_1} - 1) \cdots (U^{g_m} - 1)
\end{equation}
Therefore $w_j$ is in the direction of one of $g_1, \ldots, g_m$.
But $w_j$ is also in the direction of one of $g_{m+1} , \ldots, g_k$ since $(U^{g_{m+1}} - 1) \cdots (U^{g_k} - 1)$ also annihilates $c'$.
This gives a contradiction.
Thus we cannot have $k>m$ and we conclude that $k=m$.
\end{proof}

Let $S$ be a finite subset of $\R^2$. 
For distinct points $p$ and $q$ in $S$ define
\begin{equation}
	\ell_{p, q}=\set{t(p-q):\ t\in \R}
\end{equation}
So $\ell_{p, q}$ is the line passing through $p$ and $q$ translated to make it pass through the origin.
For a point $v\in S$ define $\dir_S(v)$ as the set $\set{\ell_{v, p}:\ p\in S, p\neq v}$.
Finally define $\dir(S)$ as $\bigcap_{v\in S} \dir_S(v)$, and we will refer to this as the \define{direction set} of $S$.
In other words, a line $\ell$ is in the direction set of $S$ if and only if any line parallel to $\ell$ is either disjoint with $S$ or intersects $S$ in at least two points.

\begin{lemma}
	\label{lemma:directions lemma}
	Let $c$ is be a binary configuration with a non-trivial annihilator.
	Then $\ord(c)\leq |\bigcap_{f\in \ann(c)}\dir(\supp(f))|$.
\end{lemma}
\begin{proof}
	Let $\ann(c) = \ab{\phi_1 \cdots \phi_m}H$ by using Theorem \ref{theorem:annihilator theorem kari szabdos}.
	Let $f\in \ann(c)$ be arbitrary and $S=\supp(f)$.
	By Theorem \ref{theorem:more annihilators from a given annihilator} and from the description of the $\phi_i$'s we see that once $v_0$ in $S$ is fixed, then each irreducible factor of $\phi_i$ divides some $U^{n_0(v-v_0)}-1$.
	From here it is easy to see that the directions appearing in the $\phi_i$'s are all members of $\dir_S(v_0)$.
	But since this is true for each $v_0$ in $S$, we have the lemma.
\end{proof}

Now suppose $F\subseteq\Z^2$ is a cluster and $c$ is an $F$-tiling of $\Z^2$.
Define the Laurent polynomial
\begin{equation}
	f(U)=\sum_{a\in F} U^{-a}
\end{equation}
Then we have $fc = k 1_{\Z^2}$, and hence for any $v\in \Z^2$ we have  $f_v: = f-U^vf$ annihilates $c$.
Note that $\bigcap_{v\in \Z^2} \dir(\supp(f_v)) \subseteq \dir(F)$.
This is because of the following reason.
If $\ell$ is a line not in $\dir(F)$, then $F$ and $v+F$ are disjoint, and the direction of $v$ is different from $\ell$.
Now if $v$ is large enough, then the direction set of $F\cup(v-F)$, which is  the support of $f-U^v f$, cannot have $\ell$ in it.
Thus by Lemma \ref{lemma:directions lemma} we have the following.

\begin{lemma}
	\label{lemma:order of a tiling}
	Let $F$ be a cluster and $c$ be an $F$-tiling of $\Z^2$.
	Then $\ord(c)\leq |\dir(F)|$.\footnote{This result applies to \emph{higher level tilings} introduced in \cite{greenfeld_tao_2020} as well.}
\end{lemma}

Therefore, if $F$ is an exact cluster with $|\dir(F)|\leq 1$, then we know that all $F$-tilings are $1$-periodic.
As a simple application, we see that all the tilings of the clusters shown in Figure \ref{figure:ell shaped cluster} and Figure \ref{figure:plus shaped cluster} are biperiodic since their direction sets are empty.

\begin{figure}[ht]
	\begin{minipage}[b]{0.45\linewidth}
	\centering
	\includegraphics[scale=0.2]{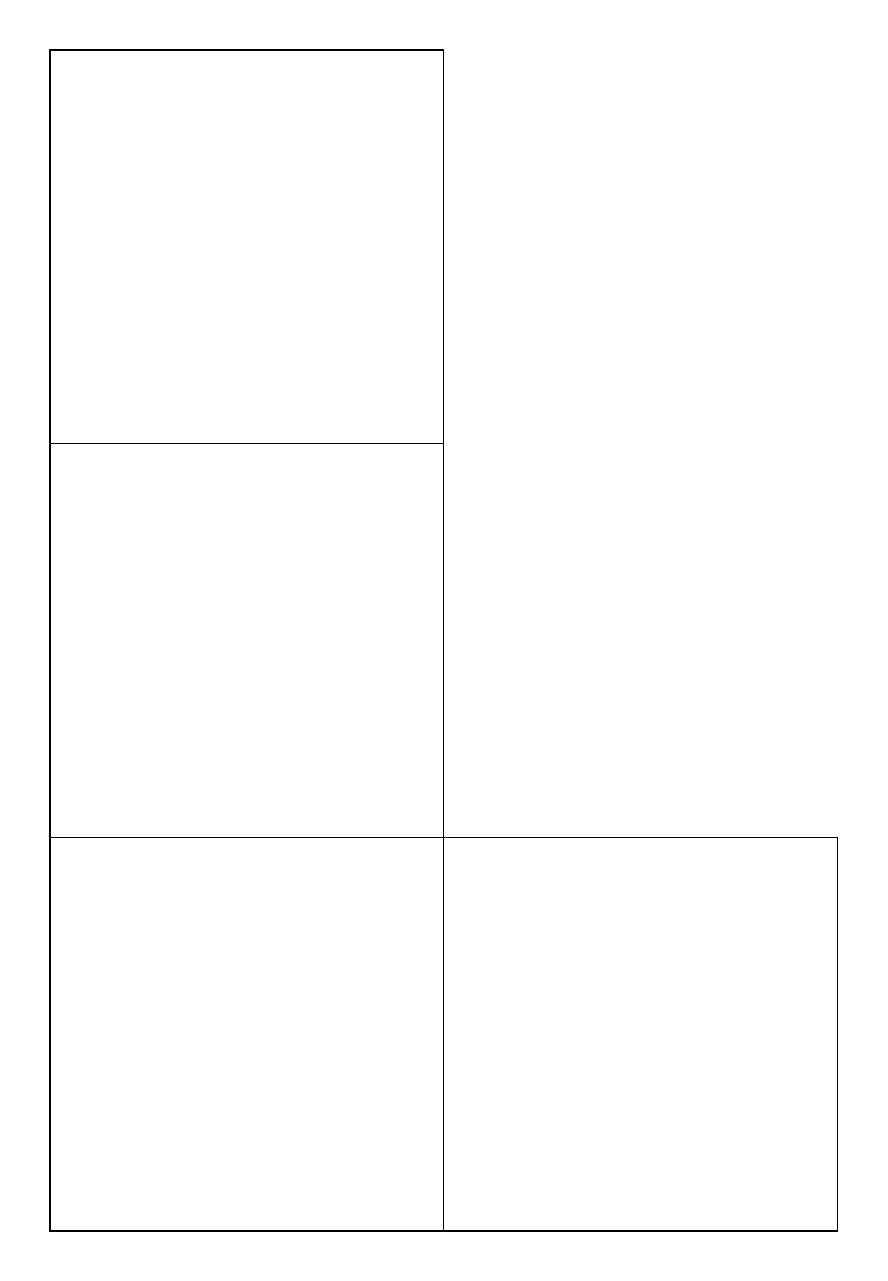}
	\caption{$L$-shaped cluster.}
	\label{figure:ell shaped cluster}
	\end{minipage}
	\hspace{0.5cm}
	\begin{minipage}[b]{0.45\linewidth}
	\centering
	\includegraphics[scale=0.2]{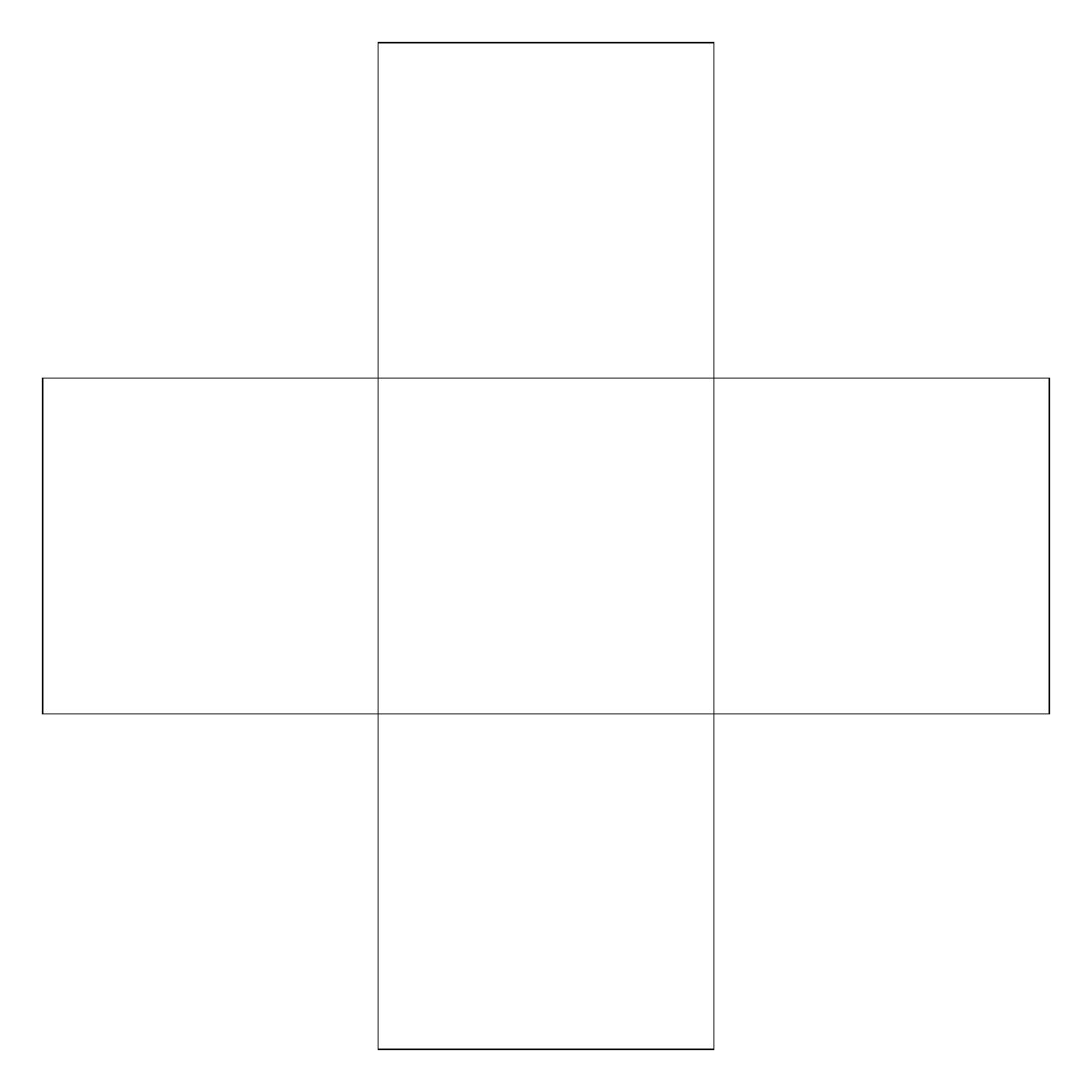}
	\caption{Plus-shaped cluster.}
	\label{figure:plus shaped cluster}
	\end{minipage}
\end{figure}
An example of an exact cluster whose direction set has cardinality 4 is the following.
\begin{figure}[H]
	\centering
	\includegraphics[scale=0.2]{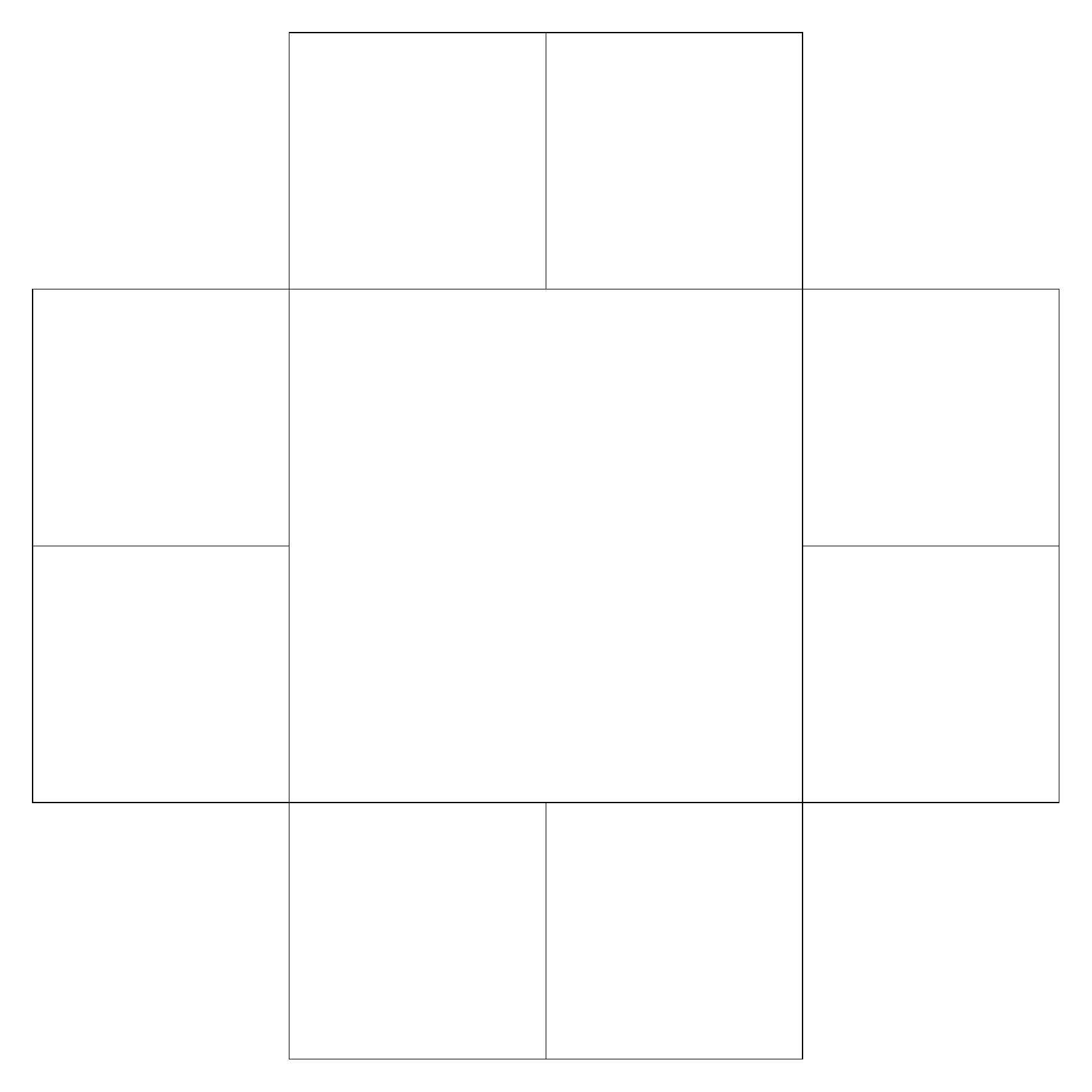}
	\caption{Octagon-shaped cluster.}
	\label{figure:octagon shaped cluster}
\end{figure}
\noindent
\section{Order Two Configurations}
\label{section:order two configuration}

\subsection{Weak Periodicity}
\label{subsection:weak periodicity}

\begin{theorem}
	\label{theorem:periodicity of binary configuration}
	Let $c$ be a binary configuration with a nontrivial annihilator.
	If $c$ has order $2$ then $c$ is weakly periodic.
\end{theorem}
\begin{proof}
	By Theorem \ref{theorem:kari szabados decomposition theorem} we can write $c=c_0+c_1+c_2$, where $c_0$ is biperiodic and $c_1$ and $c_2$ are $1$-periodic in linearly independent directions.
	Without loss of generality we may assume that $c_1$ is periodic in the direction parallel to $x$-axis and $c_2$ is periodic in the direction parallel to the $y$-axis.
	Let $\Lambda$ be a finite index subgroup of $\Z^2$ under the action of which $c_0$ is invariant.
	Let $n$ be an integer such that $ne_1$ and $ne_2$ both lie in $\Lambda_0$.
	Then we have $(U^{ne_2} - 1)(U^{ne_1} - 1)c=0$.
	Fix integers $k$ and $\ell$ and define $d:\Z^2\to \set{0, 1}$ as $d(a, b) = c(k+na, \ell+nb)$.
	Then $(U^{e_2} - 1)(U^{e_1} - 1)d=0$.
	We claim that $d$ is $1$-periodic.
	If $d(\cdot, q)$ is constant for all $q$ then $d$ is $e_1$-invariant.
	In the other case, there is an integer $q$ and an integer $p$ such that $d(p, q) \neq d(p+1, q)$.
	It follows, since $d$ takes values in $\set{0, 1}$, that $d(p, \cdot)$ and $d(p+1, \cdot)$ are both constants.
	Inductively one can now show that $d(p, \cdot)$ is constant for all $p$ and hence $d$ is $e_2$-invariant.
	
	The above argument shows that $c$ is either $ne_1$ or $ne_2$-invariant on every coset of $n \Z\times n \Z$, and hence $c$ is weakly periodic, and can be written as the sum of two $1$-periodic binary configurations.
\end{proof}

\begin{corollary}
	\label{corollary:corollary to order two fact}
	Let $F$ be a cluster with $|\dir(F)|\leq 2$ and $T$ be an $F$-tiling.
	Then $T$ can be partitioned into at most two $1$-periodic subsets.

\end{corollary}
\begin{proof}
	This follows immediately from Lemma \ref{lemma:order and directions of weakly periodicity} and Theorem \ref{theorem:periodicity of binary configuration}.
\end{proof}

We now show that natural generalizations of Theorem \ref{theorem:periodicity of binary configuration} do not hold, even after passing to the orbit closure.
Before we proceed let us review some terminology from symbolic dynamics.
A subset $S$ of $\Z$ is said to be \define{syndetic} if there is a positive integer $N$ such that for all integers $a$ we have $\set{a, a+1, \ldots, a+N-1}$ has a non-empty intersection with $S$.
A subset $S$ of $\Z^2$ is said to be \define{syndetic} if there is a positive integer $N$ such that every ball of radius $N$ in $\Z^2$ has a empty intersection with $S$.

Let $X$ be a compact metric space equipped with a continuous $\Z$ action.
We say that a point $x\in X$ is \define{uniformly recurrent} if the set
\begin{equation}
	R(x, U) = \set{n\in \Z:\ n\cdot x\in U}
\end{equation}
is syndetic for all neighborhoods $U$ of $x$ in $X$.
A similar definition of uniform recurrence can be given for a continuous $\Z^2$ action on $X$.

Given a $\Z$-action on $X$, we say that a subset $A$ of $X$ is \define{invariant } if $n\cdot x\in A$ for all $x\in A$ (and similarly for $\Z^2$-action).
An action of $\Z$ or $\Z^2$ on $X$ is said to be \define{minimal} if there is no non-empty proper closed invariant subset of $X$.
It is a fact that an action of $\Z$ or $\Z^2$ is minimal on $X$ if and only if every point of $X$ is uniformly recurrent.

Any binary configuration $c$ gives rise to a continuous action of $\Z^2$ on the orbit closure of $c$ in $\set{0, 1}^{\Z^2}$.
We say that $c$ is \define{minimal} if the action of $\Z^2$ on the orbit closure of $c$ is minimal.
This is equivalent to the uniform recurrence of $c$ in $\set{0, 1}^{\Z^2}$.
\subsection{Counterexample for Finitary Configurations}
\label{subsection:counterexample for finitary configurations}

We prove the existence of a minimal configuration $c:\Z^2\to \set{0, 1, 2}$ having order $2$ such that $c$ is not weakly periodic (by weak periodicity of $c$ we mean that each fiber of $c$ is weakly periodic).
Let $\vp, \psi\in \set{0, 1}^{\Z}$ be arbitrary.
Now define $c_{\vp, \psi}:\Z^2\to \set{0, 1, 2}$ as
\begin{equation}
	c_{\vp, \psi}(i, j)
	=
	\twopartdef{\vp(i)}{\psi(j) = 0}{\vp(i) + 1}{\psi(j) = 1}
\end{equation}
It is clear that every row of $(U^{e_2} - 1)c_{\vp, \psi}$ is constant, and hence $(U^{e_1} - 1)(U^{e_2} - 1)c_{\vp, \psi} = 0$.
Therefore $\ord(c_{\vp, \psi})$ is at most $2$.
Let $\mu$ be the uniform probability measure on $\set{0, 1}$, and $\mu^\Z$ be the corresponding product measure on $\set{0, 1}^{\Z}$.
Now given a nonzero vector $g$ in $\Z^2$, the probability that $c_{\vp, \psi}$ is constant on $\set{ng:\ n\in \Z}$ is $0$.
Thus, since $\Z^2$ is countable, for $\mu^\Z$-almost all $\vp$ and $\psi$ we have $c_{\vp, \psi}$ is not weakly periodic.
Lastly, it is also easy to argue that $c_{\vp, \psi}$ is uniformly recurrent, and hence its orbit closure is minimal with probability $1$.
We conclude that $c_{\vp, \psi}$ has order $2$ and does not have a weakly periodic point in its orbit closure with probability $1$, giving an abundance of the required counterexample.

\subsection{Counterexample for Binary Configurations with Order 3.}
\label{subsection:counterexample for binary configurations}

We show the existence of a minimal binary configuration $c$ which is not weakly periodic and has order $3$.
We use a construction given in \cite[Section 5]{kari_szabados_multi}.
Let $\alpha\in \R$ be irrational and define $c^{(1)}, c^{(2)}, c^{(3)}:\Z^2\to\set{0, 1}$ as
\begin{equation}
	c^{(1)}(i, j) = \floor{j\alpha},
	\quad
	c^{(2)}(i, j) = \floor{i\alpha},
	\quad 
	c^{(3)}(i, j) = \floor{(i+j)\alpha}
\end{equation}
Now define $c = c^{(3)} - c^{(1)} - c^{(2)}$ and note that $c$ is a binary configuration.
It was shown in \cite{kari_szabados_multi} that $c$ is not weakly periodic.

We argue that $c$ is uniformly recurrent.
Note that $c(i, j) = \floor{i\alpha+\set{j\alpha}} - \floor{i\alpha}$, where $\set{x}$ denotes the fractional part of $x$.
Let $N$ be an arbitrary positive integer and $R_N$ be the rectangle $[-N, N]\times [-N, N]$ in $\Z^2$.
We want to show that the pattern $c|_{R_N}$ induced on $R_N$ by $c$ appears in $c$ syndetically.
By the ergodicity of the irrational circle rotation, we see that $n\alpha$ comes arbitrarily close to $\alpha$ in $\R/\Z$ syndetically as $n$ varies over $\Z$, and thus the pattern that $c$ induces on $R_N$ is same as the pattern $c$ induces on $(n, 0) + R_N$ syndetically as $n$ varies over $\Z$.
Similarly, the pattern that $c$ induces on $R_N$ is same as the pattern $c$ induces on $(0, m) + R_N$ syndetically as $m$ varies over $\Z$.
Putting these two observations together one sees that $c|_{R_N}$ repeats syndetically in $c$.
Since this is true for all positive integers $N$, we deduce that $c$ is uniformly recurrent.

Thus the order of $c$ must be greater that $2$ since otherwise $c$, being uniformly recurrent, would be weakly periodic by Theorem \ref{theorem:periodicity of binary configuration}.
Also, $(U^{e_1} -1)(U^{e_2} - 1)(U^{e_1-e_2} - 1)$ is an annihilator of $c$.
Thus $c$ has order $3$ and hence is a required counterexample.

\section{Clusters of Cardinality the Square of a Prime}
\label{section:clusters of prime squared cardinality}

In \cite{szegedy_algorithms_to_tile} Szegedy showed that if $F$ is an exact cluster of prime cardinality in $\Z^2$ then every $F$-tiling is either $1$-periodic or biperiodic. (Szegedy proved periodicity results in higher dimensions as well).
In \cite[Example 4]{kari_szabados_alg_geom} this result was reproved by using polynomial techniques.

In this section we prove that if $F\subseteq \Z^2$ has cardinality of the form $p^2$, where $p$ is a prime, then the orbit closure of any $F$-tiling has a point of order at most $2$, and hence, by Corollary \ref{corollary:corollary to order two fact} it can be written as a sum of at most two $1$-periodic configurations.
Thus this result can be viewed as an extension of Szegedy's result (in two dimensions) up to passing to orbit closure.
First we collect some facts that will be used in the proof.

\subsection{Bhattacharya's Correspondence Principle}

In \cite{bhattacharya_tilings} Bhattacharya proved the periodic tiling conjecture (See \cite{lagarias_wang_tiling}) in two dimensions by developing a correspondence principle which allows to transfer the problem to an ergodic theoretic setting.
We discuss the relevant definitions and state the correspondence principle here which we will use later.

Let $X$ be a subshift of $\set{0, 1}^{\Z^2}$ and $\mu$ be a $\Z^2$-invariant probability measure on $X$.
The action of $\Z^2$ on $X$ gives an action of $\Z^2$ on $L^2(X, \mu)$:
for $\vp\in L^2(X, \mu)$ and $g\in \Z^2$, we have
\begin{equation}
	(g\cdot \vp)(x) = \vp(g\cdot x)
\end{equation}
for all $x\in X$.
In fact, $L^2(X, \mu)$ can be thought of as a $\C[U^\pm]$ module as follows.
For $\alpha(U) = \sum_{v} a_vU^v$ in $\C[U^\pm]$ and $\vp\in L^2(X, \mu)$, we have
\begin{equation}
	\alpha\cdot \vp = \sum_{v} a_v (v\cdot \vp)
\end{equation}
Thus, in particular, $U^v\vp = v\cdot \vp$.

We say that an element $\vp\in L^2(X, \mu)$ is \define{$1$-periodic} if there is $v\neq 0$ in $\Z^2$ such that $v\cdot \vp=\vp$.\footnote{We emphasize that this equation is written in $L^2$ and hence, when $f$ is an actual function, it only says that $f$ and $v\cdot f$ agree almost everywhere and not necessarily everywhere.}
An element $\vp\in L^2(X ,\mu)$ is called \define{biperiodic} if there exist two linearly independent vectors $u$ and $v$ in $\Z^2$ such that $v\cdot \vp=u\cdot \vp = \vp$.
A measurable subset $B$ of $X$ will be called $1$-periodic if $1_B$ is $1$-periodic and is called biperiodic if $1_B$ is biperiodic.
Finally, we define a measurable subset $B$ of $X$ to be \define{weakly periodic} if there exists a partition $B=B_1\sqcup \cdots \sqcup B_k$ of $B$ into finitely many measurable subsets $B_1, \ldots, B_k$ such that each $B_i$ is $1$-periodic.
Write $A=\set{x\in X:\ x(0, 0) = 1}$.

\begin{lemma}
	\emph{\cite[Section 2]{bhattacharya_tilings}}
	\label{label:bhattacharya correspondence principle}
	\textit{\textbf{Bhattacharya's Correspondence Principle.}}
	If $A$ is weakly periodic ($1$-periodic), then $\mu$-almost every point in $X$ is weakly periodic ($1$-periodic).
\end{lemma}

\subsection{A Digression: Order Versus Support of the Spectral Measure}
\label{subsection:order versus support of spectral measure}

Let $c$ be a binary configuration with a non-trivial annihilator.
Following \cite{bhattacharya_tilings}, let $X$ be the orbit closure of $c$ in $\set{0,1}^{\Z^2}$ and let $A\subseteq X$ be defined as
\begin{equation}
	A = \set{x\in X:\ x(0, 0) = 1}
\end{equation}
Equip $X$ with a $\Z^2$-ergodic measure $\mu$ and let $f= 1_A$.
\begin{lemma}
	\label{lemma:local to global}
	Let $\phi(U) \in \C[U^\pm]$ be a polynomial that annihilates $c$.
	Then $\phi$ annihilates $f$ too.
\end{lemma}
\begin{proof}
	Let $\phi=\sum_{v} a_v U^v$, where only finitely many of the $a_v$ are nonzero.
	Then we want to show that $\phi\cdot f=0$.
	Let $c'$ be an arbitrary member of in the orbit of $c$.
	We will show that $\phi\cdot f$ vanishes at $c'$.
	Let $g\in \Z^2$ be such that $c'=g\cdot c = U^gc$.
	Then
	\begin{equation}
		(\phi\cdot f)c'
		=
		\lrb{(\sum_{v}a_vU^v)f}(g\cdot c)
		=
		\sum_{v} a_vf((v+g) \cdot c)
		=
		\sum_{v} a_v c(-g - v)
		=
		(\phi\cdot  c)(-g)
		= 0
	\end{equation}
	Now we have the desired result by using the continuity of $\phi\cdot f$.
\end{proof}

Let $\nu$ be the spectral measure on $\T^2$ corresponding to the function $f/\norm{f}_2$ as discussed in Section \ref{section:spectral theorem}.
In \cite[Lemma 3.2]{bhattacharya_tilings} Bhattacharya showed that there is a finite set $\Delta$ in $\Z^2\setminus\set{0}$ such that the support of $\nu$ is contained in the union of the kernels of the characters corresponding to the elements in $\Delta$.
This is the analytic analog of the notion of order as defined by Kari and Szabados in \cite{kari_szabados_alg_geom}.
Indeed, we make this analogy precise by means of the following lemma.

\begin{lemma}
	\label{lemma:bhattachrya kernel lemma}
	If $c$ has order $m$, then there exist pairwise linearly independent vectors $g_1, \ldots, g_m$ in $\Z^2$ such that
	\begin{equation}
		\supp(\nu)\subseteq \ker\chi_{g_1}\cup \cdots \cup\ker\chi_{g_m}
	\end{equation}
\end{lemma}
\begin{proof}
	By Theorem \ref{theorem:kari szabados decomposition theorem} we can write $c$ as
	\begin{equation}
		c = c_1 + \cdots + c_m + c'
	\end{equation}
	where each $c_i$ is $1$-periodic (possibly complex valued) and $c'$ is biperiodic (again, possibly complex valued).
	Let $g_1, \ldots, g_m$ be nonzero periodicity vectors for $c_1, \ldots, c_m$.
	By scaling the $g_i$'s if necessary, we may assume that each $g_i$ fixes $c'$.
	Since the order of $c$ is $m$, we must have that the $g_i$'s are pairwise linearly independent. 
	It follows that $(U^{g_1} - 1) \cdots (U^{g_m} - 1) c = 0$, and hence, by Lemma \ref{lemma:local to global}, we have
	\begin{equation}
		(U^{g_1}-1) \cdots (U^{g_m} - 1) f = 0
	\end{equation}
	Applying the spectral theorem, we get that
	\begin{equation}
		(\chi_{g_1} - 1) \cdots (\chi_{g_m} - 1) = 0
	\end{equation}
	in $L^2(\T^2, \nu)$, whence the desired result is immediate.
\end{proof}

\begin{corollary}
	\label{corollary:very wekaly periodic decomposition}
	If $c$ has order $2$ then there exist linearly independent vectors $g$ and $h$ in $\Z^2$ such that
	\begin{equation}
		f = f_0 + f^g + f^h
	\end{equation}
	where $f_0$ is biperiodic and $f^g$ and $f^h$ are the orthogonal projections of $f$ onto the space of $g$ and $h$-invariant vectors in $L^2(X, \mu)$ respectively.
\end{corollary}
\begin{proof}
	This is immediate from \cite[Theorem 3.3]{bhattacharya_tilings} and Lemma \ref{lemma:bhattachrya kernel lemma}.\footnote{The idea is to pass to the $L^2(\T^2, \nu)$ world by applying the spectral theorem and proving the corresponding statement there.}
\end{proof}

\subsection{Preparatory Lemmas}
\label{subsection:elementary results}
\label{subsection:preparatory results}

\begin{lemma}
	\label{lemma:divisibility rectangle lemma}
	Let $p$ be a prime and $S$ be a subset of $\Z^2$ of cardinality $p^2$ such that there exist two distinct lines $l$ and $m$ passing thorough the origin such that whenever a line is parallel to either $l$ or $m$, it intersects $S$ in a set of cardinality divisible by $p$.
	Then $S$ is a rectangle.
\end{lemma}
\begin{proof}
	Without loss of generality assume that $l$ is in the direction of $(1, 0)$ and $m$ is in the direction of $(0, 1)$.
	Let $b\in \Z$ be the smallest integer such that the line $\set{(x, b)\in \Z^2:\ x\in \Z}$ intersects $F$.
	By translating $F$ if necessary, we may assume that $b=0$.
	Let $a$ be the smallest integer such that $(a, 0)$ is in $F$, and again, we may assume that $a=0$.
	Define
	\begin{equation}
		A=F\cap\set{(m, 0):\ m\in \Z} \text{ and } B = F\cap \set{(0, n):\ n\in \Z}
	\end{equation}
	We claim that $F=A\times B$.
	Let $A=\set{(a_0, 0), , \ldots, (a_{kp-1}, 0)}$, where $0=a_0< \cdots <a_{kp-1}$, and $k$ is a positive integer.
	Similarly, let $B=\set{(0, b_0), \ldots, (0, b_{\ell p-1})}$, where $0=b_0< \cdots < b_{\ell p-1}$ and $\ell$ is a positive integer.
	The sets
	\begin{equation}
		F\cap \set{(a_i, n):\ n\in \Z}, i=0, \ldots, kp-1
	\end{equation}
	are pairwise disjoint and each of these has size divisible by $p$.
	Thus the size of $F$ is at least $kp^2$, forcing $k=1$.
	Similarly $\ell=1$.
	This also shows that $F$ is contained in $A\times \Z$.
	Similarly, $F$ is contained in $\Z\times B$.
	Therefore $F$ is contained in $A\times B$.
	Since $F$ and $A\times B$ have the same cardinality, we have $F=
	A\times B$.
\end{proof}

An element $\gamma\in S^1$ is said to be \define{irrational} if it is of the form $e^{i\theta}$, there $\theta$ is an irrational multiple of $2\pi$.
Equivalently, $\gamma\in S^1$ is irrational if there is no non-trivial character of $S^1$ in whose kernel $\gamma$ lies.
More generally, elements $\gamma_1, \ldots, \gamma_m$ in $S^1$ are called \define{rationally independent} if there is no non-trivial character $\chi$ of $(S^1)^m$ such that $\chi(\gamma_1, \ldots, \gamma_m) = 1$.

\begin{lemma}
	\label{lemma:non measure preparatory lemma}
	Let $\gamma_1, \ldots, \gamma_n$ be irrational elements in $S^1$ and $x_1, \ldots, x_n$ be complex numbers.
	Assume that
	\begin{equation}
		(\gamma_1^k - 1) x_1 + \cdots + (\gamma_n^k - 1) x_n \in \Z
	\end{equation}
	for all non-negative integers $k$.
	Then the above expression is $0$ for all $k$.
\end{lemma}
\begin{proof}
	Let $m$ be the size of a maximal rationally independent subset of $\set{\gamma_1, \ldots, \gamma_n}$.
	By renumbering the $\gamma_i$'s and $x_i$'s if required, we may assume that $\set{\gamma_1, \ldots, \gamma_m}$ is a maximal rationally independent subset of $\set{\gamma_1, \ldots, \gamma_n}$.
	Thus we can find vectors $v_1, \ldots, v_n$ in $\Z^m$ such that
	\begin{equation}
		\label{equation:expression for mu}
		\gamma_1^{v_i(1)} \cdots \gamma_m^{v_i(m)} = \gamma_i
	\end{equation}
	for $i=1, \ldots, n$.
	where $v(i)$ denotes the $i$-th coordinate of any $v\in \Z^m$.
	Define the Laurent polynomial $f(U)$ in $\C[U_1^\pm, \ldots, U_m^\pm]$ as
	\begin{equation}
		f(U) = (U^{v_1} - 1)x_1 + \cdots + (U^{v_n} - 1)x_n
	\end{equation}
	Then we have
	\begin{equation}
		\label{equation:hook equation density}
		f(\gamma_1^k, \ldots, \gamma_m^k)
		=
		(\gamma_1^k - 1)x_1 + \cdots + (\gamma_n^k - 1)x_n
	\end{equation}
	By the rational independence of $\gamma_1, \ldots, \gamma_m$, we know that the set $\set{(\gamma_1^k, \ldots, \gamma_m^k):\ k\geq 0}$ is dense in $(S^1)^m$.\footnote{See Theorem 4.14 in \cite{einsiedler_ward_ergodic_theory}.}
	Thus the image of $f$ on a dense set of $(S^1)^m$ is contained in $\Z$.
	Since $f$ is continuous, this implies that the image of $(S^1)^m$ under $f$ is contained in $\Z$.
	The connectedness of $(S^1)^m$ now yields that the image of $(S^1)^m$ under $f$ is a singleton.
	However, for any $\varepsilon>0$, we can find a $k$ such that $|\gamma_i^k-1|<\varepsilon$ for each $1\leq i\leq n$, and thus we must have that $f$ is identically zero on $(S^1)^m$, finishing the proof.
\end{proof}

\begin{lemma}
	\label{lemma:kernel of character final lemma}
	Let $h=(a, b)$ be a primitive vector in $\Z^2$ and $n$ be a positive integer.
	Then
	\begin{equation}
		\ker(\chi_{n h})
		=
		\lrset{
			\lrp{
				\frac{ka}{n(a^2+b^2)} - bt, \ \frac{kb}{n(a^2+b^2)}+at
			}
			:\ t\in \R/\Z, k\in \Z
		}
	\end{equation}
\end{lemma}
\begin{proof}
	See Appendix \ref{section:elementary fact about the kernel of characters} for a proof.
\end{proof}
\subsection{Our Periodicity Result}
\label{subsection:our periodicity result}
\begin{lemma}
	\label{lemma:dilation lemma for higher level tilings}
	\textit{\textbf{Dilation Lemma.}}
	\emph{\cite[Corollary 11]{horak_kim_algebraic_method}}
	Let $F\subseteq \Z^2$.
	If $T$ is an $F$-tiling then for all $\alpha$ relatively prime with $|F|$ we have $T$ is also an $\alpha F$-tiling, where $\alpha F=\set{\alpha a:\ a\in F}$.
\end{lemma}

It should be noted that various authors (\cite{tijdeman_decomposition_of}, \cite{szegedy_algorithms_to_tile}, \cite{kari_szabados_alg_geom}, \cite{bhattacharya_tilings}, \cite{greenfeld_tao_2020}) had discovered the above lemma in one form or another.

\begin{lemma}
	\label{lemma:technical precursor to main}
	Let $F$ be a cluster containing the origin.
	Suppose $g_0$ is a nonzero vector in $\Z^2$ such that $\ker(\chi_{g_0})$ has infinitely many points which satisfy each of the equations
	\begin{equation}
		\label{equation:partition equations}
		\sum_{g\in F} \chi_{\alpha g} = 0, \quad \alpha \textup{ coprime to } p
	\end{equation}
	Then for any line $\ell$ in $\R^2$ that is parallel to the direction of $g_0$, the number of points in the intersection of $F$ and $\ell$ is divisible by $p$.
\end{lemma}
\begin{proof}
	Let $h=(a, b)$ be a primitive vector and $n$ be a positive integer such that $g_0=nh$.
	By Lemma \ref{lemma:kernel of character final lemma} we have
	\begin{equation}
		\ker(\chi_{g_0}) =
		\lrset{
				\frac{k}{n(a^2+b^2)} h + (- bt, at)\in (\R/\Z)^2 :\
			t\in \R/\Z, 0\leq k\leq n(a^2+b^2) - 1
		}
	\end{equation}
	Therefore, we can find a rational $r$ such that infinitely many elements of the set
	\begin{equation}
		\set{rh+(-bt, at)\in \R^2:\ 0\leq t< 1}
	\end{equation}
	satisfy Equations \ref{equation:partition equations}. 
	Let $r=c/d$, where $c$ and $d$ are relatively prime integers and write $v=(-b, a)$.
	We get, for each $\alpha$ coprime to $p$, 
	\begin{equation}
		\sum_{g\in F} e^{2\pi i r\alpha \ab{g, h}}e^{2\pi i \alpha \ab{g, v} t} = 0
	\end{equation}
	for infinitely many $t\in [0, 1)$.
	Therefore, the (Laurent) polynomial
	\begin{equation}
		\label{equation:intermediate equation}
		\sum_{g\in F} e^{2\pi i r \alpha \ab{g, h}} z^{\alpha \ab{g, v}}
	\end{equation}
	is satisfied by infinitely many $z\in S^1$ for each $\alpha$ coprime to $p$.
	Let $\beta$ be coprime to $p$ such that $d/\beta$ is a power of $p$ and say $d/\beta = p^m$ for some $m\geq 0$.
	Then, by Equation \ref{equation:intermediate equation} the polynomial
	\begin{equation}
		\sum_{g\in F} e^{2\pi i r \beta \ab{g, h}} z^{\beta \ab{g, v}}
		=
		\sum_{g\in F} e^{2\pi i \ab{g, h} c/p^m} z^{\beta \ab{g, v}}
	\end{equation}
	has infinitely many solutions in $S^1$ and is hence identically zero. 
	Define an equivalence relation $\sim_v$ on $F$ by writing $g_1\sim_v g_2$ for $g_1, g_2\in F$ if $\ab{g_1, v}=\ab{g_2, v}$.
	Let $F_1, \ldots, F_l$ be all the equivalence classes in $F$.
	The coefficients of the above polynomial are
	\begin{equation}
		\label{equation:coefficients of the polynomial}
		\sum_{g\in F_j} e^{2\pi i \ab{g, h} c/p^m}, \quad j=1, \ldots, l.
	\end{equation}
	and hence each of these terms are $0$.
	If $m=0$ then this cannot happen and hence we must have $m\geq 1$.
	This implies that $p$ divides $d$ and hence, since $c$ is relatively prime to $d$, we have $c$ is coprime to $p$.
	Let $\zeta$ denote the complex number $e^{2\pi i/p^m}$.
	By Equation \ref{equation:coefficients of the polynomial} we have $\zeta$ is a root of each of the (Laurent) polynomials
	\begin{equation}
		\sum_{g\in F_j} z^{\ab{g, h}c}, \quad j= 1, \ldots, l
	\end{equation}
	Therefore each of these polynomials are divisible by the $p^m$-th cyclotomic polynomial $\Phi_{p^m}(z) = \Phi_p(z^{p^{m-1}})$.
	Since $c$ is coprime to $p$, we conclude that, in fact, the polynomials
	\begin{equation}
		\sum_{g\in F_j} z^{\ab{g, h}}, \quad j= 1, \ldots, l
	\end{equation}
	are all divisible by $\Phi_p(z^{p^{m-1}})$.
	Say $\sum_{g\in F_j} z^{\ab{g, h}} = \Phi_p(z^{p^{m-1}})\Psi_j(z)$ for some integer (Laurent) polynomial $\Psi_j(z)$.
	Substituting $z=1$ in particular shows that each $|F_j|$ is divisible by $p$ and we are done.
\end{proof}

\begin{theorem}
	\label{theorem:prime square theorem}
	Let $F$ be an exact cluster with cardinality $p^2$, where $p$ is a prime and $T$ be an $F$-tiling.
	Then the orbit closure of $T$ has a point of order at most $2$.
\end{theorem}
\begin{proof}
	We recall some notations.
	Let $X\subseteq \set{0, 1}^{\Z^2}$ be the set of all the $F$-tilings and $\mu$ be a $\Z^2$-ergodic probability measure on $X$, which we may assume to be concentrated on the orbit closure of $T$.
	We define $A=\set{x\in X:\ x(0, 0) = 1}$ and $f\in L^2(X, \mu)$ as the characteristic function on $A$.
	Let $\nu$ be the spectral measure associated to the unit vector $f/\norm{f}_2$ in $L^2(X, \mu)$ and $\theta$ be the unitary isomorphism between $H_f$ --- the span closure of the orbit of $f/\norm{f}_2$ --- and $L^2(\T^2, \nu)$ as discussed in Theorem \ref{theorem:spectral theorem}.
	\emph{We will use $\T$ to denote the unit circle and not $\R/\Z$.}

	Let $\Delta\subseteq \mathbb Z^2\setminus \set{0}$ be of minimum possible size such that the spectral measure $\nu$ associated to $f/\norm{f}_2$ is supported on $\bigcup_{g\in \Delta}\ker(\chi_g)$.
	We know such a $\Delta$ exists by Lemma \ref{lemma:bhattachrya kernel lemma}.
	Consider the equations
	\begin{equation}
		\sum_{g\in F} \chi_{\alpha g} = 0, \quad \alpha \textup{ coprime to } p
	\end{equation}
	If there exist two distinct members $g_0$ and $g_1$ in $\Delta$ such that infinitely many members of $\ker(\chi_{g_i})$, $i=0, 1$, satisfy each of the equations above, then by Lemma \ref{lemma:technical precursor to main} we would have two distinct lines $\ell$ and $\ell'$ passing though the origin such that if $\ell''$ is a line parallel to either $\ell$ or $\ell'$, then the number of points of $F$ that $\ell''$ passes through is divisible by $p$.
	This combined with Lemma \ref{lemma:divisibility rectangle lemma} shows that $F$ is a rectangle.
	Since the direction set of a rectangle has size at most $2$, we deduce by Lemma \ref{lemma:order of a tiling} that the order of any $F$-tiling is most $2$.

	So we may assume that there is at most one $g_0$ in $\Delta$ with the property that infinitely many elements of $\ker(\chi_{g_0})$ are solutions to each of the following equations
	\begin{equation}
		\sum_{g\in F} \chi_{\alpha g} = 0,\quad \alpha \text{ coprime  to } p
	\end{equation}
	Applying Lemma \ref{lemma:dilation lemma for higher level tilings} we have every element of $X$ is an $(\alpha F)$-tiling, whenever $\alpha$ is coprime to $p$. 
	Therefore
	\begin{equation}
		\sum_{g\in F} 1_{\alpha g A} = 1_X
	\end{equation}
	whenever $\alpha$ is coprime to $p$.
	The function $1_X$ is invariant under each $g\in \Z^2$, and hence so is its image under $\theta$.
	The only elements of $L^2(\T^2, \nu)$ that are invariant under each element of $\Z^2$ are the functions in the span of $\delta_{(1, 1)}$, the Dirac delta function concentrated at $(1, 1)$.
	Therefore we get from the previous equation that
	\begin{equation}
		\label{equation:hook}
		\sum_{g\in F} \chi_{\alpha g} = c\delta_{(1, 1)} \text{ for all } \alpha \text{ coprime to } p
	\end{equation}
	in $L^2(\T^2, \nu)$, where $c$ is a constant. 
	This shows that
	\begin{equation}
		\label{equation:a restriction on support of the spectral measure}
		\supp(\nu) \subseteq \lrb{\bigcap_{\alpha: \gcd(\alpha, p) = 1} Z\lrp{\sum_{g\in F} \chi_{\alpha g}}} \cup \set{(1, 1)}
	\end{equation}
	where $Z(\sum_{g\in F} \chi_{\alpha g})$ denotes the set of points in $\T^2$ where $\sum_{g\in F} \chi_{\alpha g}$ vanishes.

	For each $h\neq g_0$ in $\Delta$, let $S_h$ be the set defined as
	\begin{equation}
		S_h 
		= \lrb{
				\bigcap_{\alpha:\ \gcd(\alpha, p) = 1} Z\lrp{\sum_{g\in F} \chi_{\alpha g}}
			}
		   	\cap \ker(\chi_h)
	\end{equation}
	and by our assumption each of these are finite.
	Since $\supp(\nu)$ is also contained in $\bigcup_{g\in \Delta} \ker\chi_{g}$, we see from Equation \ref{equation:a restriction on support of the spectral measure} that the set
	\begin{equation}
		\ker(\chi_{g_0}) \cup \bigcup_{h\in \Delta:\ h\neq g_0} S_h
	\end{equation}
	is a $\nu$-full measure set.
	
	So we infer that there is a finite set $S\subseteq \T^2$ such that $\nu$ is supported on $\ker(\chi_{g_0})\cup S$.
	We may assume that $S$ is disjoint with $\ker(\chi_{g_0})$ and that each element of $S$ has positive mass under $\nu$.
	Further, we may assume that $S$ has smallest size with this property.
	This implies that $\chi_{g_0}(s)$ is irrational for all $s\in S$, for otherwise we could replace $g_0$ by a scale of itself and reduce the size of $S$.

	We will show that $S$ is empty.
	Assume on the contrary that $S$ is non-empty.
	Define an equivalence relation $\sim$ on $S$ by writing $p\sim q$ for $p, q\in S$ if $\chi_{g_0}(p)=\chi_{g_0}(q)$.
	Let $\mc E$ be the set of all the equivalence classes.
	Thus
	\begin{equation}
		1 = 1_{\ker(\chi_{g_0})} + \sum_{E\in \mc E} 1_E
	\end{equation}
	in $L^2(\T^2, \nu)$.
	Choose a representative $p_E\in E$ for each $E\in \mc E$.
	Now acting both sides of the above equation by $ng_0$, where $n$ is any non-negative integer, we get
	\begin{equation}
		\chi_{n g_0} = 1_{\ker(\chi_{g_0})} + \sum_{E\in \mc E}\chi_{g_0}(p_E)^n 1_E
	\end{equation}
	Going back into the $L^2(X, \mu)$ world by applying $\theta^{-1}$, we get that
	\begin{equation}
		\label{equation:hook super equation}
		(n g_0)\cdot f = f^{g_0} + \sum_{E\in \mc E} \chi_{g_0}(p_E)^n \vp_E
	\end{equation}
	where $\vp_E = \norm{f}_2 \theta^{-1}(1_E)$, and $f^{g_0}$ is the orthogonal projection of $f$ onto the space of $g_0$-invariant functions $L^2(X, \mu)$.\footnote{By the Birkhoff ergodic theorem we see that $f^{g_0}$ lies in $H_f$.
	Also, $1_{\ker\chi_{g_0}}$ is the orthogonal projection of $1$ onto the space of $g_0$-invariant functions in $L^2(\T^2, \nu)$. The fact that $\theta$ is a unitary isomorphism shows that $\theta^{-1}(1_{\ker\chi_{g_0}})$ is same as $f^{g_0}/\norm{f}_2$.}
	Therefore, for each $n\geq 0$, we have
	\begin{equation}
		(ng_0)\cdot f - f = \sum_{E\in \mc E} (\chi_{g_0}(p_E)^n - 1)\vp_E
	\end{equation}
	Let $Y\subseteq X$ be a $\mu$-full measure subset of $X$ such that
	\begin{equation}
		[(ng_0)\cdot f](y) - f(y) = \sum_{E\in \mc E}(\chi_{g_0}(p_E)^n - 1) \vp_E(y)
	\end{equation}
	for all $y\in Y$ and all $n\geq 0$.
	Therefore
	\begin{equation}
		\sum_{E\in \mc E} (\chi_{g_0}(p_E)^n - 1) \vp_E(y) \in \Z
	\end{equation}
	for all $y\in Y$, and all $n\geq 0$.
	But since each $\chi_{g_0}(p_E)$ is irrational, we may apply Lemma \ref{lemma:non measure preparatory lemma} to deduce that for any $y\in Y$ the only value the above expression can take, for any non-negative integer $n$, and hence in particular for $n=1$, is $0$.
	Therefore $[g_0 \cdot f] (y) - f(y)$ is $0$ for each $y\in Y$.
	But since $Y$ is a full measure, we infer that $g_0\cdot f = f$ in $L^2(X, \mu)$.
	Thus $f$ is $1$-periodic, and hence by \hyperref[label:bhattacharya correspondence principle]{Bhattacharya's Correspondence Principle} we deduce that the orbit closure of $T$ has a $1$-periodic point in it.
	Every $1$-periodic configuration has order $1$, and this concludes the proof.
\end{proof}

\begin{corollary}
	\label{corollary:corollary to prime square theorem}
	Let $F$ be an exact cluster with cardinality $p^2$, where $p$ is a prime and $T$ be an $F$-tiling.
	Then there is a point in the orbit closure of $T$ that can be written as the sum of at most two $1$-periodic binary configurations.
\end{corollary}
\begin{proof}
	Immediate from Theorem \ref{theorem:periodicity of binary configuration} and Theorem \ref{theorem:prime square theorem}.
\end{proof}

\appendix

\section{An Elementary Fact About the Kernel of Characters}
\label{section:elementary fact about the kernel of characters}

\begin{lemma}
	\label{lemma:kernel of character first lemma}
	Let $(a, b)$ be a primitive vector in $\Z^2$ and $(x, y)\in \R^2$ be such that $ax+by\in \Z^2$.
	Then there is $t\in \R$ such that $(x, y)\equiv (-b t, a t)\pmod{\Z^2}$.
\end{lemma}
\begin{proof}
	Since $(a, b)$ is primitive, there exists $(x_0, y_0)\in \Z^2$ such that $ax_0+by_0=1$.
	Let $d=ax+by$ and define $t=-xy_0+yx_0$.
	Then we have
	\begin{equation}
		x+bt=x+b(-xy_0+yx_0) = x-bxy_0+byx_0 = x-bxy_0+(d-ax)x_0 = x-x(ax_0+by_0) + dx_0 = dx_0
	\end{equation}
	Similarly $y-at = dy_0$.
   Since $d$ is an integer, we have the lemma.
\end{proof}

\begin{lemma}
	\label{lemma:kernel of character second lemma}
	Let $h=(a, b)$ be a primitive vector in $\Z^2$.
	Then $\ker(\chi_h)$ is precisely the set
	\begin{equation}
		\set{(-bt, at):\ t\in \R/\Z}
	\end{equation}
\end{lemma}
\begin{proof}
	Let $(x, y)\in (\R/\Z)^2$ be arbitrary.
	Then $(x, y)$ is in $\ker(\chi_h)$ if and only if $ax+by=0$ in $\R/\Z$.
	By Lemma \ref{lemma:kernel of character first lemma} we know that there is $t\in \R/\Z$ such that $(x+by, y-at) = (0, 0)$  in $(\R/\Z)^2$, that is, $x=-bt$ and $y=at$.
	This shows $\ker(\chi_h)$ is contained in the set $\set{(-bt, at):\ t\in \R/\Z}$. 
	The reverse containment is easily checked.
\end{proof}

\begin{lemma}
	\label{lemma:kernel of character third lemma}
	Let $h=(a, b)$ be a primitive vector in $\Z^2$.
	Then we have
	\begin{equation}
		\ker(\chi_h)
		=
		\lrset{
			\lrp{
				\frac{ka}{a^2+b^2} - bt, \ \frac{kb}{a^2+b^2}+at
			}
			:\ t\in \R/\Z, k\in \Z
		}
	\end{equation}
\end{lemma}
\begin{proof}
	It is easy to see that $\ker(\chi_h)$ contains the set described above.
	We show the containment in the other direction.
	By Lemma \ref{lemma:kernel of character second lemma} the points in $\R^2$ which map into $\ker(\chi_h)$ under the natural projection $\R^2\to (\R/\Z)^2$ are precisely the points of the form $(m, n) + (-bs, as)$, where $s\in \R$ is arbitrary.
	Now any point in $\R^2$ can be written as $\lambda(a, b) + t(-b, a)$ for some $\lambda, t\in \R$.
	For $(m, n) + (-bs, as)$ the corresponding $\lambda$ is $(ma+nb)/(a^2+b^2)$.
	Since $ma+nb$ can range over all integers as $m$ and $n$ range over the integers, we see that every point $(m, n) + (-bs, as)$ is of the form
	\begin{equation}
		\frac{k(a, b)}{a^2+b^2} + (-bt, at)
	\end{equation}
	for some integer $k$, and conversely.
	This proves the lemma.
\end{proof}

\begin{proof}[Proof of Lemma \ref{lemma:kernel of character final lemma}]
	Let $(x, y)$ be in $\ker(\chi_{nh})$.
	Then $(nx, ny)$ is in $\ker\chi_h$.
	The result now follows from Lemma \ref{lemma:kernel of character third lemma}.
\end{proof}

\bibliographystyle{alpha}
\bibliography{UBOP.bib}

\end{document}